\documentclass[12pt]{amsart}
\usepackage{amsmath,amssymb,amsthm}
\usepackage{latexsym,ulem}      
\usepackage{colordvi,graphicx}
\usepackage{bm} 
\usepackage{url}
\numberwithin{equation}{section}

\newtheorem{theorem}{Theorem}[section]

\newtheorem{lemma}[theorem]{Lemma}
\newtheorem{corollary}[theorem]{Corollary}
\theoremstyle{remark}
\newtheorem{example}[theorem]{Example}
\newtheorem{remark}[theorem]{Remark}
\newtheorem{definition}[theorem]{Definition}
\newcounter{FNC}[page]
\def\fauxfootnote#1{{\addtocounter{FNC}{2}\Magenta{$^\fnsymbol{FNC}$}%
     \let\thefootnote\relax\footnotetext{\Magenta{$^\fnsymbol{FNC}$#1}}}}
\newcommand{\C}{{\mathbb C}}

\newcommand{\calX}{{\mathcal X}}
\newcommand{\calY}{{\mathcal Y}}

\newcommand{\Fl}{{\mathbb F}\ell}
\newcommand{\adot}{{a_\bullet}}
\newcommand{\apdot}{{a_\bullet^\perp}}
\newcommand{\bdot}{{b_\bullet}}
\newcommand{\Edot}{E_{\bullet}}
\newcommand{\Fdot}{F_{\bullet}}
\newcommand{\Fpdot}{F_{\bullet}^\perp}

\DeclareMathOperator{\Gr}{Gr}

\newcommand{\bw}{{\bf w}}

\newcommand{\sspan}[1]{\langle\langle #1\rangle\rangle}

\newcommand{\defcolor}[1]{\Blue{#1}}
\newcommand{\demph}[1]{\defcolor{{\sl #1}}}

\title[A lifted square formulation for certifiable Schubert calculus]{A lifted square
  formulation for\\certifiable Schubert calculus}   
\author{Nickolas Hein, Frank Sottile}
\address{Nickolas Hein \\
         Department of Mathematics\\
         University of Nebraska at Kearney\\
         Kearney\\
         Nebraska \ 68849\\
         USA}
\email{heinnj@unk.edu}
\urladdr{http://www.unk.edu/academics/math/faculty/About\_Nickolas\_Hein/}
\address{Frank Sottile \\
         Department of Mathematics\\
         Texas A\&M University\\
         College Station\\
         Texas \ 77843\\
         USA}
\email{sottile@math.tamu.edu}
\urladdr{http://www.math.tamu.edu/~sottile/}
\thanks{Research of Hein and Sottile supported in part by NSF grant DMS-0915211.}
\begin{document}
\begin{abstract}
 Formulating a Schubert problem as the solutions to a system of equations in
 either Pl\"ucker space or in the local coordinates of a Schubert cell
 usually involves more equations than variables.
 Using reduction to the diagonal, we previously gave a primal-dual formulation for Schubert problems that  
 involved the same number of variables as equations (a square formulation).
 Here, we give a different square formulation by lifting incidence conditions which typically involves fewer
 equations and variables. 
 Our motivation is certification of numerical computation using Smale's $\alpha$-theory.
\end{abstract}
\maketitle

A $m\times n$ matrix $M$ with $m\geq n$ is rank-deficient if and only if all of its
$n\times n$ minors vanish.
This occurs if and only if there is a nonzero vector $v\in\C^n$ with $Mv=0$.
There are $\binom{m}{n}$ minors and each is a polynomial of degree $n$ in the $mn$ entries of $M$.
In local coordinates for $v$, the second formulation gives $m$ bilinear equations in $mn{+}n{-}1$
variables, and the map $(M,v)\mapsto M$ is a bijection over an open dense set of matrices of rank $n{-}1$.
The set of rank-deficient matrices has dimension $(m{+}1)(n{-}1)$, which shows that the second formulation is
a complete intersection, while the first is not if $m>n$.
The principle at work here is that adding extra information may simplify the description of a degeneracy locus.

Schubert varieties in the flag manifold are universal degeneracy loci~\cite{Fu92}.
We explain how to add information to a Schubert variety to simplify its description in local coordinates.
This formulates membership in a Schubert variety as a complete intersection of bilinear equations
and  formulates any Schubert problem as a square system of bilinear equations.
This lifted formulation is both different from and typically significantly more efficient than the primal-dual
square formulation of~\cite{HaHS}, as we demonstrate in Section~\ref{S:compare}. 

Our motivation comes from numerical algebraic geometry~\cite{SW05}, which uses 
numerical analysis to represent and manipulate algebraic varieties on a computer.
It does this by solving systems of polynomial equations and following solutions along curves.
For numerical stability, low degree polynomials are preferable to high
degree polynomials.
More essential is that Smale's $\alpha$-theory~\cite{Smale86} enables the certification of computed solutions to
square systems of polynomial equations~\cite{alphaC}, and therefore efficient square formulations of systems of
polynomial equations are desirable. 
Furthermore, the estimates used in implementations of $\alpha$-theory simplify for bilinear systems, as
explained in~\cite[Rem.~2.11]{HaHS}.
Interestingly, formulations as square systems of bilinear equations may also aid Gr\"obner basis
computations.
Faug\`ere, et al.~\cite{FSS} gave improved complexity bounds for
zero-dimensional bilinear systems.

The Schubert calculus is a well-understood, rich family of enumerative problems which has served as a
laboratory to study new phenomena in enumerative geometry~\cite{MS}.
Problems in Schubert calculus lead to highly-structured systems of polynomials that are challenging to study.
Traditional formulations of most problems in Schubert calculus are not complete intersections, and those which are
complete intersections have far fewer solutions than predicted by the BKK bound~\cite{BKK}---this is
demonstrated in Table~2 of~\cite{So00b}.

Square formulations of Schubert problems also 
enable the certified computation of monodromy, using either the algorithm of Beltr{\'a}n and
Leykin~\cite{BL12,BL13} or the Newton homotopies of Hauenstein and Liddell~\cite{HL14}.
This will in turn enable the certified computation of Galois groups~\cite{LS,MS}.
Because general degeneracy loci are pullbacks of Schubert varieties, these square formulations may
lead to formulations of  more general problems involving degeneracy loci as square systems of polynomials.

In Section~\ref{S:traditional} we explain the traditional formulation of Schubert problems using Stiefel
coordinates and determinantal equations expressing rank conditions.
In Section~\ref{S:lifted} we give our new lifted square formulation for Schubert varieties and Schubert
problems, illustrating with some examples.
In Section~\ref{S:compare} we compare the efficiency of the lifted formulation with the
primal-dual formulation of~\cite{HaHS}, demonstrating that the lifted formulation typically
involves fewer equations and variables, and through three examples that computations using
it consume fewer resouorces.

\section{Determinantal formulation of Schubert problems}\label{S:traditional}

The Schubert calculus involves all questions of determining the (flags of) linear subspaces of a vector space
that have specified positions with respect to other (fixed, but general) flags of linear subspaces.
We briefly sketch the Schubert calculus, Stiefel coordinates for Schubert varieties, and
traditional determinantal formulations of Schubert problems.
We work over the complex numbers for convenience and motivation from numerical algebraic geometry.
Our formulations and main result are valid over any field, if we replace claims of transversality by properness
(expected dimension) when the field has positive characteristic.
This is because Kleiman's result showing transversality of the intersection of general translates becomes
properness in positive characteristic~\cite{Kl74}.
For the Grassmannian, we retain transversality as Vakil~\cite{Va06b} proved that
general translates of Schubert varieties in a Grassmannian intersect transversally in any characteristic.

%
\subsection{Schubert problems}
Fix an integer $n$ and a sequence $\adot\colon a_1<\dotsb<a_s<n$ of positive integers.
A \demph{flag of type $\adot$} is a sequence of linear subspaces
\[
   \Edot\ \colon\ E_{a_1}\ \subset\ E_{a_2}\ \subset\ \dotsb\ \subset\ E_{a_s}\ \subset\ \C^n\,,
\]
where $\dim E_{a_j}=a_j$.
A flag is \demph{complete} if $\adot=\{1,2,\dotsc,n{-}1\}$.
Given a flag $\Edot$ of type $\adot$, there is a list $(e_1,\dotsc,e_{a_s})$ 
of independent vectors such that $E_{a_j}$ is the linear span of 
$\{e_1,\dotsc,e_{a_j}\}$ for each $1\leq j\leq s$.
In this case, write $\Edot=\defcolor{\sspan{e_1,\dotsc,e_{a_s}}_{\adot}}$.
The set of all flags of type $\adot$ is an algebraic manifold \defcolor{$\Fl(\adot;n)$} of dimension 
 \begin{equation}\label{Eq:dimFlagVariety}
   \defcolor{\dim(\adot)}\ :=\ \sum_{j=1}^s (n-a_j)(a_j-a_{j-1})
   \ =\ n\cdot a_s \ -\ \sum_{i=1}^s a_j(a_j-a_{j-1})\,,
 \end{equation}
where $a_0:=0$.
When $s=1$, the flag manifold $\Fl(\adot;n)$ is the \demph{Grassmannian} of $a_1$-planes in $\C^n$,
\defcolor{$\Gr(a_1;n)$}, which has dimension $a_1(n{-}a_1)$.  

The \demph{position} of a flag $\Edot$ of type $\adot$ with respect to a complete flag $\Fdot$ is the 
$n\times s$ array of nonnegative integers $\dim(F_i\cap E_{a_j})$ for $i=1,\dotsc,n$ and $j=1,\dotsc,s$.
These positions are encoded by permutations $w\in S_n$ with descents in $\adot$.
For such a permutation $w$, $w(i)>w(i{+}1)$ implies that $i=a_j$, for some $j$.
Write \defcolor{$W^\adot$} for this set of permutations.
Given $w\in W^\adot$ and a complete flag $\Fdot$, we have the 
\demph{Schubert cell}, 
 \begin{equation}\label{Eq:SchubertCell}
   \defcolor{X^\circ_w \Fdot}\ :=\ \{\Edot\in\Fl(\adot)\ |\ \dim(F_i\cap E_{a_j})=
    \#\{ k\leq a_j \,\mid\, w(k)\leq i\} \}\,.
 \end{equation}
The \demph{Schubert variety} \defcolor{$X_w\Fdot$} is the closure of $X^\circ_w \Fdot$ and is
obtained by replacing the dimension equality in~\eqref{Eq:SchubertCell} with an inequality $\geq$.
This has dimension $\defcolor{\ell(w)}:=\#\{k<j\mid w(k)>w(j)\}$, the number of inversions of $w$, and thus
codimension $\defcolor{|w|}:=\dim(\adot)-\ell(w)$.  

A \demph{Schubert problem} is a list $\defcolor{\bw}:=(w_1,\dotsc,w_r)$ of elements $w_i\in W^{\adot}$ for
$i=1,\dotsc,r$ satisfying $|w_1|+\dotsb+|w_r|=\dim(\adot)$. 
Given a Schubert problem $\bw$, Kleiman showed~\cite{Kl74} there is an open dense subset of the product of flag
manifolds consisting of $r$-tuples of flags $(\Fdot^1,\dotsc,\Fdot^r)$ such that the intersection
 \[
   X_{w_1}\Fdot^1\;\cap\;
   X_{w_2}\Fdot^2\;\cap\;\dotsb\;\cap\;
   X_{w_r}\Fdot^r
 \]
is transverse.
Kleiman's Theorem implies that the points of intersection lie in the corresponding Schubert
cells---we lose nothing (for general flags) if we restrict to Schubert cells, and the same reasoning allows
us to restrict to any dense open subset of the Schubert varieties.
The number of points in the intersection is independent of the choice of general flags and this number may be
determined by algorithms in the Schubert calculus.

\subsection{Determinantal formulation of a Schubert variety}

Suppose that $\calX$ is a set of $n\times a_s$ matrices $x$ whose column vectors $e_1(x),\dotsc,e_{a_s}(x)$ are
independent.
The association 
\[
   \calX\ni x\ \longmapsto\ \sspan{e_1(x),\dotsc,e_{a_s}(x)}_\adot\ =:\ \defcolor{\Edot(x)}
\]
defines a map $\calX\to\Fl(\adot;n)$.
We call $\calX$ \demph{Stiefel coordinates} for the closure of the image of this map.
We have $\defcolor{E_{a_j}(x)}:=\mbox{span}\{e_1(x),\dotsc,e_{a_j}(x)\}$, and we also write $E_{a_j}(x)$ for the 
$n\times a_j$ matrix whose columns are $e_1(x),\dotsc,e_{a_j}(x)$.
Whether we intend the subspace or the matrix will be clear from context.

Suppose that a set $\calX$ of $n\times a_s$ matrices forms Stiefel coordinates for some subset
$X\subset\Fl(\adot;n)$.   
Let $\Fdot$ be a flag with a basis $f_1,\dotsc,f_n$ that forms the columns of a $n\times n$ matrix.
Write \defcolor{$F_k$} both for the $k$-dimensional subspace of the flag $\Fdot$ and for the $n\times k$
matrix with columns $f_1,\dotsc,f_k$.

For $w\in W^{\adot}$, set $\defcolor{r_{i,j}(w)}:=\#\{ k\leq a_j \,\mid\, w(k)\leq i\}$.
Then $\Edot\in X_w\Fdot$ if and only if 
\[
   \dim F_i\cap E_{a_j}\ \geq\ r_{i,j}(w)\qquad i=1,\dotsc,n\quad j=1,\dotsc,s\,.
\]
Then the condition on $x\in\calX$ that $\Edot(x)\in X_w\Fdot$ is
 \[  
   \mbox{rank}\, \bigl(\, F_i\,\mid\, E_{a_j}(x)\,\bigr)\ \leq\ 
    i+a_j-r_{i,j}(w)\qquad i=1,\dotsc,n\quad j=1,\dotsc,s\,.
  \] 
This is given by the vanishing of minors of $(F_i\mid E_{a_j})$ of size $i+a_j-r_{i,j}(w)+1$.
These are polynomials in the entries of $x\in\calX$.
Not all such minors are needed.
Even if redundant minors are eliminated, the number that remains will in general exceed $|w|$.
This is discussed for Grassmannians in Section 1.3 of~\cite{HaHS}, where it is shown that 
after removing redundancy, $|w|=0,\,1$ or $a_1=1,\,n{-}1$ are the only cases for which this number of
minors equals $|w|$. 
In Section~\ref{S:comparison} we present a typical example minimally requiring 17 minors, but where $|w|=4$.

\subsection{Stiefel coordinates for Schubert varieties}\label{SS:Steifel}

A Schubert cell $X^\circ_w\Fdot$ has a description in terms of bases.
For a flag $\Fdot$, set $F_0:=\{0\}$.

\begin{lemma}\label{L:SCell_Span}
 Let $w\in W^\adot$ and $\Fdot$ be a complete flag.
 Then a flag $\Edot$  of type $\adot$ lies in $X^\circ_w\Fdot$ if and only if there exist vectors 
 $(e_1,\dotsc,e_{a_s})$ with  $e_k\in F_{w(k)}\smallsetminus F_{w(k)-1}$ for $k=1,\dotsc,a_s$ and 
 $\Edot=\sspan{e_1,\dotsc,e_{a_s}}_{\adot}$.
\end{lemma}

\begin{proof}
 Let $e_1,\dotsc,e_{a_s}$ be such a collection of vectors with $e_k\in F_{w(k)}\smallsetminus F_{w(k)-1}$.
 As $w$ is a permutation, these vectors are linearly independent.
 If $\Edot=\sspan{e_1,\dotsc,e_{a_s}}_{\adot}$, then 
\[
   F_i\cap E_{a_j}\ =\ 
    \langle e_k\,\mid\, k\leq a_j\mbox{ and } w(k)\leq i\rangle\,,
\]
 and so $\Edot\in X^\circ_w\Fdot$.
 Conversely, if $\Edot\in X^\circ_w\Fdot$, observe that if $a_{j-1}< k\leq a_j$, then 
 the condition that $\Edot$ lies in the Schubert cell and $w\in W^{\adot}$  implies that
\[
   \dim F_{w(k)}\cap E_{a_j}\ =\ 1\ +\ \dim F_{w(k)-1}\cap E_{a_j}\,.
\]
 For each such $j$ and $k$, let $e_k$ be a nonzero vector that, together with $F_{w(k)-1}\cap E_{a_j}$, 
 spans $F_{w(k)}\cap E_{a_j}$.
 Then $e_k\in F_{w(k)}\smallsetminus F_{w(k)-1}$, and $\Edot=\sspan{e_1,\dotsc,e_{a_s}}_{\adot}$.
\end{proof}

Lemma~\ref{L:SCell_Span} leads to the usual Stiefel coordinates for Schubert cells~\cite[Ch.~10]{Fulton}.
Given $w\in W^{\adot}$, let $\calX_w$ be the collection of $n\times a_s$ matrices $(x_{i,j})$
such that 
 \begin{eqnarray*}
  x_{w(k),k} &=& 1\mbox{ for }k=1,\dotsc,a_s\\
  x_{i,j}&=&0 \mbox{ if } i>w(j)\mbox{ or }i=w(k)\mbox{ for some }k<j\,,
 \end{eqnarray*}
and $x_{i,j}$ is otherwise unconstrained.
For example, here are typical matrices in $\calX_w$ for $w=5724613$ with $\adot=(2,5)$
and $w=3652471$ with $\adot=(2,3,5,6)$,
\[
  \left(\begin{array}{ccccc}
        x_{1,1}&x_{1,2}&x_{1,3}&x_{1,4}&x_{1,5}\\
        x_{2,1}&x_{2,2}&   1   &   0   &   0   \\
        x_{3,1}&x_{3,2}&   0   &x_{3,4}&x_{3,5}\\
        x_{4,1}&x_{4,2}&   0   &   1   &   0   \\
           1   &   0   &   0   &   0   &   0   \\
           0   &x_{6,2}&   0   &   0   &   1   \\
           0   &   1   &   0   &   0   &   0   \end{array}\right)
   \qquad
  \left(\begin{array}{cccccc}
        x_{1,1}&x_{1,2}&x_{1,3}&x_{1,4}&x_{1,5}&x_{1,6}\\
        x_{2,1}&x_{2,2}&x_{2,3}&   1   &   0  &   0   \\
           1   &   0   &   0   &   0   &   0   &   0   \\
           0   &x_{4,2}&x_{4,3}&   0   &   1   &   0   \\
           0   &x_{5,2}&   1   &   0   &   0   &   0   \\
           0   &   1   &   0   &   0   &   0   &   0   \\
           0   &   0   &   0   &   0   &   0   &   1   \end{array}\right)
\]
It is an exercise to show that if $x\in\calX_w$, then $\Edot(x)\in X^\circ_w\Fdot$, where $\Fdot$
is the standard coordinate flag in $\C^n$. 
Suppose that $\Edot\in X^\circ_w\Fdot$ and $\Edot=\sspan{e_1,\dotsc,e_{a_s}}_{\adot}$ as in Lemma~\ref{L:SCell_Span}. 
Let $y$ be a $n\times a_s$ matrix with column vectors $e_1,\dotsc,e_{a_s}$.
If we reduce each column of $y$ modulo those to its left, we obtain a matrix in $\calX_w$.
We summarize this discussion.

\begin{lemma}
 For any $w\in W^\adot$, the set $\calX_w$ gives Stiefel coordinates for the Schubert variety
 $X_w\Fdot$ where $\Fdot$ is the standard coordinate flag.
 The map $\calX_w\to X_w\Fdot$ defined by $x\mapsto \Edot(x)$ is a bijection between $\calX_w$ and the
 Schubert cell $X_w^\circ\Fdot$.
\end{lemma}

An entry $(i,j)$ is unconstrained for matrices in $\calX_w$ when $i<w(j)$ and there is no $k<j$ with
$i=w(k)$.
As $w$ is a permutation, there is some $k>j$ with $i=w(k)$.
Thus the unconstrained entries in $\calX_w$ correspond to inversions in the permutation $w$, and so we conclude
that $\dim\calX_w=\ell(w)$, the number of inversions in $w$.

\subsection{Determinantal formulation of a Schubert problem}\label{SS:traditional}

Let $\bw=(w_1,\dotsc,w_r)$ be a Schubert problem and suppose that $\Fdot^1,\dotsc,\Fdot^r$ are general complete
flags.
Choosing a basis for $\C^n$, if necessary, we may assume that $\Fdot^1$ is the standard coordinate flag.
Let $\Fdot^2,\dotsc,\Fdot^r$ be $n\times n$ matrices corresponding to the flags of the same name.
Then, in the local Steifel coordinates $\calX_{w_1}$ for $X_{w_1}\Fdot^1$, the instance
\[
   X_{w_1}\Fdot^1\,\cap\,
   X_{w_2}\Fdot^2\,\cap\, \dotsb \,\cap\,
   X_{w_r}\Fdot^r
\]
of the Schubert problem is given by the rank conditions
 \[
   \mbox{rank}\, \bigl(\, F_i^k\,\mid\, E_{a_j}(x)\,\bigr)\ \leq\ 
    i+a_j - r_{i,j}(w_k)
 \]
for $i=1,\dotsc,n$, $j=1,\dotsc,s$, and $k=2,\dotsc,r$.
These rank conditions are equivalent to the vanishing of minors of appropriate sizes of these matrices.
As we discussed, this typically involves more equations than variables.
Call this the \demph{determinantal formulation} of the Schubert problem.

%
\section{Lifted square formulations for Schubert problems}\label{S:lifted}

We give a new formulation for Schubert varieties as complete intersections in that the number of variables is
equal to the sum of the dimension of the Schubert variety and the number of equations. 
These equations are bilinear when we use Stiefel coordinates for the flag manifold.
This leads to a square formulation of any Schubert problem.
In Subsection~\ref{SS:improved} we explain an improvement to this formulation.

Fix a sequence $\adot\colon a_1<\dotsb<a_s<n$ and let $\Edot$ be a flag of
type $\adot$ in $\C^n$.
A complete flag $\Fdot$ in $\C^n$ induces complete flags on each quotient vector space $E_{a_j}/E_{a_{j-1}}$ for 
$j=1,\dotsc,s$. 
The subspaces in the induced flag on $E_{a_j}/E_{a_{j-1}}$ are
 \begin{equation}\label{Eq:induced_flag}
    \bigl( (E_{a_j}\cap F_k) \ +\ E_{a_{j-1}}\bigr)/ E_{a_{j-1}}
    \qquad\mbox{\rm for } k=1,\dotsc,n\,.
 \end{equation}
If $w$ is the unique permutation in $W^{\adot}$ such that $\Edot\in X^\circ_w\Fdot$, so that $\Edot$ and
$\Fdot$ have relative position $w$, then the
subspaces~\eqref{Eq:induced_flag} in the flag on
$E_{a_j}/E_{a_{j-1}}$ induced by $\Fdot$ are 
\[
     \bigl((E_{a_j}\cap F_{w(k)}) \ +\ E_{a_{j-1}}\bigr)/ E_{a_{j-1}}
    \qquad\mbox{\rm for } a_{j-1}<k\leq a_j\,.
\]  
(Recall that $a_{j-1}<i<k\leq a_j$ implies that $w(i)<w(k)$ and thus $F_{w(i)}\subset F_{w(k)}$.)
When $\Edot=\sspan{e_1,\dotsc,e_{a_s}}_{\adot}$ for independent vectors $e_1,\dotsc,e_{a_s}$, we have 
another complete flag in each quotient space $E_{a_j}/E_{a_{j-1}}$ for $j=1,\dotsc,s$,
whose subspaces are
 \begin{equation}\label{Eq:other_flag}
     \bigl(\langle e_{k}, e_{k+1},\dotsc,e_{a_j}\rangle \ +\ E_{a_{j-1}}\bigr)/ E_{a_{j-1}}
    \qquad\mbox{\rm for } a_{j-1}<k\leq a_j\,.
 \end{equation}
We say that $(e_1,\dotsc,e_{a_s})$ and $\Fdot$ are in \demph{$\adot$-general position} if for each
$j=1,\dotsc,s$, the two flags~\eqref{Eq:induced_flag} and~\eqref{Eq:other_flag} on  $E_{a_j}/E_{a_{j-1}}$
are in linear general position.
That is, an intersection $G\cap H$ of subspaces, one from each
flag, has the expected dimension $\dim G + \dim H - \dim(E_{a_j}/E_{a_{j-1}})$.

The set of those $(e_1,\dotsc,e_{a_s})$ with $\Edot=\sspan{e_1,\dotsc,e_{a_s}}_{\adot}$ that are in
$\adot$-general position with $\Fdot$ forms an open and dense subset of those $(e_1,\dotsc,e_{a_s})$ with
$\Edot=\sspan{e_1,\dotsc,e_{a_s}}_{\adot}$.
Indeed, there is a dense open subset of the general linear group giving linear combinations of 
the sublist $e_{a_{j-1}+1},\dotsc,e_{a_j}$ which induce a flag on $E_{a_j}/E_{a_{j-1}}$ in linear general
position with the flag induced by $\Fdot$.

%
\subsection{Lifted square formulation}\label{SS:lifted}
The lifted square formulation relies upon the following lemma.
For a number $k\leq a_s$, define $\defcolor{\lceil k\rceil_{\adot}}:=\min\{a_j\mid k\leq a_j\}$, which is the
smallest number in $\adot$ that is at least as large as $k$.

\begin{lemma}\label{L:square}
 Suppose that $\Edot=\sspan{e_1,\dotsc,e_{a_s}}_{\adot}$ is a flag of type $\adot$, $\Fdot$ is a complete
 flag in $\adot$-general position with $(e_1,\dotsc,e_{a_s})$, and $w\in W^{\adot}$.
 Then $\Edot\in X_w^\circ\Fdot$ if and only if for each $k=1,\dotsc,a_s$ there are numbers
 $\alpha_{k,i}$ for $i\leq\lceil k\rceil_{\adot}$ with $w(k)<w(i)$ such that 
 \begin{equation}\label{Eq:newBasis}
     g_k\ :=\ e_k\ +\ \sum_{\substack{i\leq\lceil k\rceil_{\adot}\\ w(k)<w(i)}} \alpha_{k,i} e_i\,,
 \end{equation}
 where $g_k\in F_{w(k)}\smallsetminus F_{w(k)-1}$.
 Furthermore these numbers $\alpha_{k,i}$ are the unique numbers with this property.
\end{lemma}

We illustrate this lemma with two examples.

\begin{example}\label{Ex:lifted_1}
 Suppose that $E_3:=\langle e_1,e_2,e_3\rangle$ lies in the Schubert cell $X_{358\,12467}^\circ\Fdot$ in the
 Grassmannian $\Gr(3;8)$ and $(e_1,e_2,e_3)$ is in general position with $\Fdot$.
 Then there are constants $\alpha_{1,2}$, $\alpha_{1,3}$, and $\alpha_{2,3}$ such that if
 \begin{eqnarray}
   g_1&:=& \makebox[97pt][r]{$e_1 + \alpha_{1,2}e_2 + \alpha_{1,3}e_3$}\,,\nonumber\\
   g_2&:=& \makebox[97pt][r]{$e_2 + \alpha_{2,3}e_3$}\,,\ \mbox{ and}\label{Eq:g_formula}\\
   g_3&:=& \makebox[97pt][r]{$e_3$}\,,\nonumber
 \end{eqnarray}
 then $g_1\in F_3$, $g_2\in F_5$, and $g_3\in F_8=\C^8$.

 View these now as variables and equations for membership in  $X_{358\,12467}\Fdot$.
 The linear forms defining the subspaces in $\Fdot$ give $5+3+0=8$ equations on the
 vectors $e_1,e_2,e_3$ and variables $\alpha_{1,2},\alpha_{1,3},\alpha_{2,3}$.
 As these linear forms are are general, they  define a subset of codimension eight which
 when projected to the Grassmannian gives a subset of codimension five, which is the
 codimension of $X_{358\,12467}\Fdot$. \hfill$\diamond$
\end{example}

\begin{example}\label{Ex:lifted_2}
 Suppose that $\Edot:=\sspan{e_1,e_2,e_3,e_4}_{2<4}$ lies in the Schubert cell
 $X_{59\,47\,12368}^\circ\Fdot$ of the flag manifold $\Fl(2,4;9)$ and $(e_1,e_2,e_3,e_4)$ is in general
 position with $\Fdot$. 
 Then there are constants $\alpha_{1,2}$, $\alpha_{3,1}$, $\alpha_{3,2}$, $\alpha_{3,4}$, and
 $\alpha_{4,2}$ such that if 
 \begin{eqnarray*}
   g_1&:=& \makebox[72.6pt][r]{$e_1 + \alpha_{1,2}e_2$}\,, \\
   g_2&:=& \makebox[72.6pt][r]{$e_2$}\,,\\
   g_3&:=& \makebox[72.6pt][r]{$\alpha_{3,1}e_1+\alpha_{3,2}e_2$} +
             \makebox[54pt][r]{$e_3 + \alpha_{3,4}e_4$}\,,\ \mbox{ and}\\
   g_3&:=& \makebox[72.6pt][r]{$\alpha_{4,2}e_2$} + \makebox[54pt][r]{$e_4$}\,,
 \end{eqnarray*}
 then $g_1\in F_5$, $g_2\in F_9=\C^9$, $g_3\in F_4$, and $g_4\in F_7$.
 As a formulation for membership in $X_{59\,47\,12368}\Fdot$, the linear forms defining
 the $F_i$ give $4+0+5+2=11$ equations on the vectors $e_1,e_2,e_3,e_4$ and five variables
 $\alpha_{k,i}$. 
 As these forms are general, they define a subset of codimension eleven that when projected to $\Fl(2,4;9)$
 gives a subset of codimension six, which is the codimension of $X_{59\,47\,12368}\Fdot$.
 Since the membership equations ($g_i\in F_{w(i)}$, etc.) are linear in the variables $\alpha_{k,i}$, the
 fibers over points of $X_{59\,47\,12368}\Fdot$ are affine spaces.
 The equality of dimensions and surjectivity implies that the fiber over a general point is a singleton, which
 is the unicity  assertion in Lemma~\ref{L:square}.
\hfill$\diamond$
\end{example}

\begin{proof}[Proof of Lemma~\ref{L:square}]
 Suppose first that $g_k\in F_{w(k)}\smallsetminus F_{w(k)-1}$ where  $g_1,\dotsc,g_{a_s}$ are defined
 using~\eqref{Eq:newBasis} for some constants $\alpha_{k,i}$.
 Then $\Edot=\sspan{g_1,\dotsc,g_{a_s}}_{\adot}$, as the 
 expressions~\eqref{Eq:newBasis} are unitriangular.
 Lemma~\ref{L:SCell_Span} then implies that $\Edot\in X^\circ_w\Fdot$.

 For the other direction, we use induction on $j$ to construct unique constants $\alpha_{k,i}$ such that
 the vector $g_k$ defined by~\eqref{Eq:newBasis} satisfies $g_k\in F_{w(k)}\smallsetminus F_{w(k)-1}$ for
 $k\leq a_j$. 
 We will suppose that that for each $k\leq a_{j-1}$ there are unique constants
 $\alpha_{k,i}$ for $i\leq\lceil k\rceil_{\adot}$ with $w(k)<w(i)$ such that if $g_k$ is the linear
 combination~\eqref{Eq:newBasis}, then $g_k\in F_{w(k)}\smallsetminus F_{w(k)-1}$, and use this to obtain the
 constants  $\alpha_{k,i}$ for $a_{j-1}<k\leq a_j$.
 This is no assumption in the base case $(j=1)$ of this construction.

 By our assumption on $(e_1,\dotsc,e_{a_s})$ and $\Fdot$, the two flags in $E_{a_j}/E_{a_{j-1}}$, 
\[
  E_{a_{j-1}}\ \subsetneq\ 
  E_{a_{j-1}}+\langle e_{a_j} \rangle\  \subsetneq\  \dotsb\  \subsetneq \ 
  E_{a_{j-1}}+\langle e_{a_{j-1}+2},\dotsc,e_{a_j}\rangle\  \subsetneq\  E_{a_j}\,,\ \mbox{and}
\]
\[
  E_{a_{j-1}}\ \subsetneq\ 
  E_{a_{j-1}}+(F_{w(a_{j-1}+1)}\cap E_{a_j})\  \subsetneq\  \dotsb\  \subsetneq \ 
  E_{a_{j-1}}+(F_{w(a_j-1)}\cap E_{a_j})\  \subsetneq\  E_{a_j}\,,
\]
 are opposite.
 In particular, for any $a_{j-1}<k,i\leq a_j$, we have that
 \begin{equation}\label{Eq:Piece}
    \Bigl(E_{a_{j-1}}+ ( F_{w(k)}\cap E_{a_j})\Bigr)\ \cap\ 
    \Bigl(E_{a_{j-1}}+\langle e_{i},\dotsc,e_{a_j}\rangle\Bigr)
 \end{equation}
 has dimension $\max(0,k{+}1{-}i)$ modulo $E_{a_{j-1}}$.
 This implies that there are constants $\alpha_{k,\ell}$ for $k<\ell\leq a_j$ and an element $e\in E_{a_{j-1}}$
 such that the sum
 \begin{equation}\label{Eq:summ}
  e_k\ +\ \sum_{\ell=k+1}^{a_j} \alpha_{k,\ell} e_\ell\ \ +\ e
 \end{equation}
 lies in $F_{w(k)}$.
 In fact, the sum~\eqref{Eq:summ} lies in $F_{w(k)}\smallsetminus F_{w(k)-1}$.
 Indeed, as $\Edot\in X^\circ_w\Fdot$, we have that 
 $F_{w(k)-1}\cap E_{a_j}\subsetneq F_{w(k)}\cap E_{a_j}$, and so the dimension of~\eqref{Eq:Piece} drops
 if we replace $F_{w(k)}$ by $F_{w(k)-1}$.
 This also implies that the numbers  $\alpha_{k,\ell}$ are unique.

 The element $e\in E_{a_{j-1}}$ is some linear combination of $e_1,\dotsc,e_{a_{j-1}}$
 and thus also of $g_1,\dotsc,g_{a_{j-1}}$.
 Since $g_i\in F_{w(i)}$, those $g_i$ with $w(i)<w(k)$ are not needed for the sum~\eqref{Eq:summ} to lie in
 $F_{w(k)}$, and thus there are constants $\beta_i$ for $i\leq a_{j-1}$ with $w(k)<w(i)$ such that 
 \begin{equation}\label{Eq:g_k}
  g_k\ :=\ e_k\ +\ \sum_{\ell=k+1}^{a_j} \alpha_{k,\ell} e_\ell\ \ +\ 
    \sum_{\substack{i\leq a_j\\ w(k)<w(i)}} \beta_i g_i
 \end{equation}
 lies in $F_{w(k)}\smallsetminus F_{w(k)-1}$.
 As each $g_i$ in the second sum lies in $F_{w(i)}\smallsetminus F_{w(i)-1}$, the constants
 $\beta_i$ are unique.
 To obtain the expression~\eqref{Eq:newBasis} for $g_k$ first use the formula~\eqref{Eq:newBasis} for
 each $g_i$ appearing in~\eqref{Eq:g_k} to rewrite the second sum as a linear combination of $e_\ell$
 for $\ell\leq a_{j-1}$ with $w(k)<w(i)<w(\ell)$, and then use that $w\in W^{\adot}$ to see that 
 $\{k{+}1,\dotsc,a_j\}$ is the set of $i$ in the interval $(a_{j-1},a_j]$ with $w(k)<w(i)$.
 The unicity of the constants $\alpha_{k,i}$ follows from that of the constants $\alpha_{k,\ell}$ and
 $\beta_i$, and our induction hypothesis.
\end{proof}

\begin{remark}\label{R:SquareFormulation}
 Lemma~\ref{L:square} leads to a square formulation for membership in $X_w\Fdot$ for flags in $\Fl(\adot;n)$ as
 follows.
 \begin{enumerate}
  \item Pick Stiefel coordinates $\calX_{\adot}$ for $\Fl(\adot;n)$.
         For $x\in\calX_{\adot}$, we have the partial flag, 
         $\Edot(x)=\sspan{e_1(x),\dotsc,e_{a_s}(x)}_{\adot}\in \Fl(\adot;n)$.
  \item Choose lifting coordinates
  \begin{equation}\label{Eq:lifting_coordinates} 
         \alpha\ =\ \{\alpha_{k,i}\mid k=1,\dotsc,a_s\,,\ 
              i\leq\lceil k\rceil_{\adot} \mbox{ with }w(i)>w(k)\}
  \end{equation}
  and form the vectors
 \[
     g_k(x,\alpha)\ :=\ e_k(x)\ +\ \sum_{\substack{i\leq\lceil k\rceil_{\adot}\\ w(k)<w(i)}} \alpha_{k,i}
     e_i(x)\,,
\]
    for $k=1,\dotsc,a_s$.
  \item Given independent linear forms $f_{1},\dotsc,f_{n}$ such that
         $F_j$ is defined by the vanishing of $f_{j+1},\dotsc,f_n$, our equations for 
         $\Edot(x)\in X_w\Fdot$ are
\[
         f_j(g_i(x,\alpha))\ =\ 0\qquad\mbox{for } i=1,\dotsc,a_s\mbox{ and }j>w(i)\,.
\]
 \end{enumerate}
  These equations are bilinear in the sets of variables $x\in\calX_{\adot}$ and $\alpha$.
 \hfill$\diamond$
\end{remark}

\begin{definition}
 Call the formulation for membership in $X_w\Fdot$ for flags in $\Fl(\adot;n)$ of
 Remark~\ref{R:SquareFormulation} the \demph{lifted formulation} for a Schubert variety.
 Write \defcolor{$\alpha(w)$} for the set of lifting coordinates~\eqref{Eq:lifting_coordinates} 
 and \defcolor{$|\alpha(w)|$} for the number of these coordinates, which is
 \begin{equation}\label{Eq:number_variables}
   |\alpha(w)|\ =\ 
     \sum_{k=1}^{a_s} \#\{ i\leq \lceil k\rceil_{\adot}\mid w(i)>w(k)\}\,.
 \end{equation}
\end{definition}

\begin{theorem}\label{Thm:CompleteIntersection}
 The lifted formulation for membership in $X_w\Fdot\subset\Fl(\adot)$ is a complete intersection. 
\end{theorem}

\begin{proof}
 We must show that $\dim X_w\Fdot$ equals the number of variables 
 minus the number of equations.
 The number of equations is the sum of codimensions of the $F_{w(k)}$ for $k\leq a_s$,
 \begin{equation}\label{Eq:NEqs}
   \sum_{k=1}^{a_s} n-w(k)\ =\  n\cdot a_s\ -\ \sum_{k=1}^{a_s} w(k)\,.
 \end{equation}
 The number of variables is the dimension of $\Fl(\adot;n)$, as calculated in~\eqref{Eq:dimFlagVariety}
 \begin{equation}\label{Eq:dim}
  \dim(\adot)\ =\  n\cdot a_s\ -\ \sum_{j=1}^s (a_j-a_{j-1})a_j\,,
 \end{equation}
 where $a_0=0$, plus the number $|\alpha(w)|$ of the variables $\alpha_{k,i}$.
 We rewrite~\eqref{Eq:number_variables} as 
 \begin{multline*}
   \quad \sum_{k=1}^{a_s} \bigl(\lceil k\rceil_{\adot}
         -\#\{ i\leq k\mid w(i)\leq w(k)\}\bigr)\\
    =\  \Bigl(\sum_{k=1}^{a_s} \lceil k\rceil_{\adot} \Bigr)
         \  -\#\{ i\leq k\leq a_s \mid w(i)\leq w(k)\}\bigr)\,.\quad
 \end{multline*}
 The first equality uses that if $a_j<i<k\leq a_{j+1}$, then $w(i)<w(k)$, as $w\in W^{\adot}$.
 We rewrite this as 
 \begin{equation}\label{Eq:NewVariables}
   \Bigl(\sum_{j=1}^s (a_j-a_{j-1})a_j\Bigr)\ \ -\ 
    \#\{ i \leq k\leq a_s \mid w(i)\leq w(k)\}\,.
 \end{equation}
 Using that $w\in W^{\adot}$, the linear combination $\eqref{Eq:NewVariables}+\eqref{Eq:dim}-\eqref{Eq:NEqs}$
 becomes  
 \begin{multline*}
  \quad
  \Bigl(\sum_{k=1}^{a_s} w(k)\Bigr)\ -\ \#\{i \leq k\leq a_s\mid w(i)\leq w(k)\} \\
    =\ \#\{i<k\mid w(i)>w(k)\}\ =\ 
      \ell(w)\ =\ \dim X_w\Fdot\,,\quad
 \end{multline*}
 which completes the proof.
\end{proof}

\begin{remark}\label{R:Square_Formulation}
 The lifted formulation of Remark~\ref{R:SquareFormulation} for a Schubert variety leads to a square
 formulation for Schubert problems, following Subsection~\ref{SS:traditional}.
 Suppose that $\bw:=(w_1,\dotsc,w_r)$ is a Schubert problem on $\Fl(\adot;n)$.
 Let $\Fdot^1,\dotsc,\Fdot^r$ be general flags and consider the intersection of Schubert varieties
 \begin{equation}\label{Eq:SchubertProblem}
   X_{w_1}\Fdot^1\,\cap\,
   X_{w_2}\Fdot^2\,\cap\, \dotsb \,\cap\,
   X_{w_r}\Fdot^r
 \end{equation}
 Assume that $\Fdot^1$ is the standard coordinate flag and use Steifel
 coordinates $\calX_{w_1}$ for the Schubert cell $X_{w_1}^\cdot\Fdot^1$ to formulate the
 intersection~\eqref{Eq:SchubertProblem}.
 Replacing the determinantal rank conditions for membership in each Schubert variety
 $X_{w_2}\Fdot^2,\dotsc,X_{w_r}\Fdot^r$ by the lifted square formulation gives
 the \demph{lifted formulation} for the Schubert problem $\bw$.
 It uses 
\[
   \ell(w_1) + |\alpha(w_2)| + \dotsb + |\alpha(w_r)|
\]
 variables and bilinear equations.
\hfill$\diamond$
\end{remark}

Since the intersection~\eqref{Eq:SchubertProblem} is transverse, it
is zero-dimensional (or empty).
This gives the following corollary to
Theorem~\ref{Thm:CompleteIntersection}.

\begin{corollary}\label{Cor:SquareXw1}
 The lifted formulation for membership in the intersection~$\eqref{Eq:SchubertProblem}$ is a complete intersection
 in the local coordinates $\calX_{w_1}$. 
\end{corollary}

\begin{remark}\label{R:ReductionXw1nXw2}
 For Grassmannians, there are Stiefel coordinates parametrizing the
 intersection $X_{w_1}\Fdot^1\,\cap\,X_{w_2}\Fdot^2$~\cite[\S~3.1]{HaHS}.
 These involve $\dim(\Gr(a_1;n))-|w_1|-|w_2|=\ell(w_1)-|w_2|$
 variables and lead to a lifted formulation of~\eqref{Eq:SchubertProblem} using
\[
   \ell(w_1) - |w_2| + |\alpha(w_3)| + \dotsb + |\alpha(w_r)|
\]
 variables and bilinear equations.
 This presents~\eqref{Eq:SchubertProblem} as a complete intersection using $|w_2|+|\alpha(w_2)|$ fewer
 equations and variables than the formulation of Corollary~\ref{Cor:SquareXw1}. 
\hfill$\diamond$
\end{remark}

\subsection{Reduced lifted formulation}\label{SS:improved}
We introduce an improvement to the lifted square formulation, motivating it through three examples.

\begin{example}\label{Ex:improved_lifted_1}
 Consider the lifted formulation for the Schubert variety $X_{w}\Fdot$ in $\Gr(3;8)$ where $w=458\,12367$.
 Suppose that $E_3=\langle e_1,e_2,e_3\rangle$ where $(e_1,e_2,e_3)$ come from Steifel coordinates $\calX$ for
 $\Gr(3;8)$ and involve 15 variables.
 The lifted formulation uses three new variables $\alpha_{1,2}$, $\alpha_{1,3}$, and $\alpha_{2,3}$ as in  
 Example~\ref{Ex:lifted_1} and we form the vectors $g_1,g_2,g_3$ as in~\eqref{Eq:g_formula}.
 Then $\Edot\in X_w\Fdot$ if and only if $g_1\in F_4$, $g_2\in F_5$, and $g_3\in F_8$,
 giving seven equations in $15+3$ variables to define the codimension four Schubert variety $X_w\Fdot$.

 It suffices to only require that $g_1$ and $g_2$ lie in $F_5$,
 for then some linear combination of the two will lie in $F_4$.
 This dispenses with one equation.
 Having done this, we may also dispense with the variable $\alpha_{1,2}$, and thereby obtain a reduction
 of one variable and one equation.
 Specifically, suppose that 
 \begin{eqnarray*}
   g_1&:=& \makebox[78.4pt][r]{$e_1 \hspace{24.8pt} + \alpha_{1,3}e_3$}\,,\\
   g_2&:=& \makebox[78.4pt][r]{$e_2 + \alpha_{2,3}e_3$}\,,\ \mbox{ and}\\
   g_3&:=& \makebox[78.4pt][r]{$e_3$}\,.
 \end{eqnarray*}
 Then $E_3\in X_{w}\Fdot$ if  $g_1,g_2\in F_5$.
 This gives six equations in $15+2$ variables to define $X_{w}\Fdot$ in the Steifel coordinates
 $\calX$ for $\Gr(3;8)$.
 \hfill$\diamond$
\end{example}

\begin{example}\label{Ex:improved_lifted_2}
 A similar reduction is possible in Example~\ref{Ex:lifted_2}.
 The requirement that $g_3\in F_4$ may be relaxed to $g_3\in F_5$, for then some linear combination of $g_3$
 and $g_1$ lies in $F_4$.
 This removes one bilinear equation, and we may dispense with $\alpha_{3,1}$.
 \hfill$\diamond$
\end{example}

\begin{example}\label{Ex:improved_lifted_3}
 Now consider the lifted formulation for $X_{w}\Fdot$ with $w=358\,47\,126$, which
 has codimension six in the 21-dimensional flag manifold $\Fl(3,5;8)$.
 Suppose that $e_1,\dotsc,e_5$ are the column vectors from the Steifel coordinates $\calX_{3<5}$ for
 $\Fl(3,5;8)$ and set $\Edot:=\sspan{e_1,\dotsc,e_5}_{3<5}$. 
 The lifted formulation uses seven variables $\alpha_{k,i}$ to form the vectors
 \begin{eqnarray*}
  g_1 &=& \makebox[97pt][r]{$e_1 + \alpha_{1,2}e_2 + \alpha_{1,3}e_3$}\,,\\
  g_2 &=& \makebox[97pt][r]{$e_2 + \alpha_{2,3}e_3$}\,,\\
  g_3 &=& \makebox[97pt][r]{$e_3$}\,,\\
  g_4 &=& \makebox[97pt][r]{$\alpha_{4,2}e_2 + \alpha_{4,3}e_3$} + 
            \makebox[54pt][r]{$e_4 + \alpha_{4,5}e_5$}\,,\ \mbox{ and}\\
  g_5 &=& \makebox[97pt][r]{$ \alpha_{5,3}e_3$} + 
             \makebox[54pt][r]{$e_5$}\,,
 \end{eqnarray*}
 which are required to lie in the subspaces
\[
   g_1\in F_3\,,\ g_2\in F_5\,,\ g_3\in F_8\,,\ g_4\in F_4\,,\ \mbox{ and }\ g_5\in F_7\,.
\]
 This gives $5+3+0+4+1=13$ bilinear equations in $|\adot|+|\alpha(w)|=21+7$ variables to define
 $X_w\Fdot$.

 Some of these conditions and variables are redundant.
 All that is needed is that $g_1\in F_3$, $g_2,g_4\in F_5$ and $g_3,g_5\in F_8$.
 The first three give $5+3+3=11$ bilinear equations and the last two give none.
 Similarly, the variables $\alpha_{4,2}$,  and $\alpha_{5,3}$ are not
 needed.
 Thus  $X_{w}\Fdot$ has a formulation involving five new variables and eleven
 bilinear equations. \hfill$\diamond$
\end{example}

The reduction in these examples was possible when for some $k<a_s$ there was a number \defcolor{$m$} such that 
the consecutive values $w(k){+}1,\dotsc,w(k){+}m$ for the permutation $w$ occured at positions
$i\leq\lceil k\rceil_\adot$.
If $i_1,\dotsc,i_m\leq\lceil k\rceil_\adot$ are the positions such that $w(i_j)=w(k)+j$, then the condition
of Lemma~\ref{L:square} that $g_k\in F_{w(k)}$ may be replaced by $g_k\in F_{w(k)+m}=F_{w(i_m)}$, for
there is some linear combination of the vectors $g_k,g_{i_1},\dotsc,g_{i_m}$ that lies in $F_{w(k)}$.
Likewise, the variables $\alpha_{k,i_1},\dotsc,\alpha_{k,i_m}$ are not needed.

We formalize this.
Consider vectors $e_1(x),\dotsc,e_{a_s}(x)$ coming from Steifel coordinates $\calX$ for some subset $X$ of
$\Fl(\adot;n)$. 
For each $k=1,\dotsc,a_s$, let $\beta_k$ be the set of indeterminates
 \begin{equation}\label{eq:beta_k}
   \defcolor{\beta_k}\ :=\ 
   \{ \beta_{k,i}\mid i\leq\lceil k\rceil_\adot\ 
       \mbox{and  $\exists\, j>\lceil k\rceil_\adot$ with}\ w(k)<w(j)<w(i)\}\,,
 \end{equation}
 and set
 \begin{equation}\label{eq:new_g_k}
   g_k\ =\ \defcolor{g_k(x,\beta)}\ :=\ 
      e_k(x)\ +\ \sum_i \beta_{k,i} e_i(x)\ .
 \end{equation}
 Write $\beta=\defcolor{\beta(w)}=\cup_k \beta_k$ for the set of all these indeterminates and
 \defcolor{$|\beta(w)|$} for the number of indeterminates in $\beta$.

 For a complete flag $\Fdot$, the \demph{reduced lifted formulation for membership in $X_w\Fdot$} in the
 Steifel coordinates $\calX$ uses the additional variables $\beta(w)$ to form the
 expressions~\eqref{eq:new_g_k}, and has the equations given by the membership requirements
\[
    g_k(x,\beta)\ \in\ F_{w(k)+m(k)}\,,\ \mbox{ for } k=1,\dotsc,a_s\,,
\]
 where $m(k)$ is the largest number $m$ such that the consecutive values 
 $w(k){+}1,\dotsc,w(k){+}m$ for the permutation $w$ occur at positions $i\leq\lceil k\rceil_\adot$.

 The results in Subsection~\ref{SS:lifted} hold {\it mutatis mutandis} for this reduced lifted formulation of
 Schubert varieties and Schubert problems and are omitted.

%
\section{Comparison with the primal-dual square formulation}\label{S:compare}

We compare the efficiency of this lifted formulation to the
primal-dual formulation of~\cite{HaHS}.
Both involve added variables and bilinear equations in local Steifel coordinates.
We first compare these formulations to the determinatal formulation of a particular Schubert variety.
Next, we determine which of the two uses fewer added variables for each Schubert variety on a flag manifold
in $\C^9$, and then compare their computational efficiency for solving three Schubert problems, 
including two from~\cite{HaHS}.
We almost always observe a gain in efficiency for the lifted formulation over the primal-dual
formulation.

We may take advantage of whichever formulation is most efficient for a given Schubert variety,
for they are compatible.
That is, one may construct a hybrid system of equations for the intersection~\eqref{Eq:SchubertProblem} using a 
lifted formulation to determine membership in some of the Schubert varieties and a primal-dual
formulation to determine membership in the others. 
Whenever $|w|=1$, the determinantal formulation for membership in the hypersurface Schubert variety $X_w\Fdot$
is a single determinant in Steifel coordinates, so there is no need to use an alternative
formulation to obtain a square system.
In what follows, we will always use the determinantal formulation when $|w|=1$.

\subsection{Comparison of three formulations}\label{S:comparison}
We compare the three formulations, determinantal, primal-dual, and lifted, for membership in the Schubert
variety $X_{3478\,1256}\Fdot$ in $\Gr(4,8)$.
Let $\calX$ be Steifel coordinates for $\Gr(4,8)$, which is a set of $8\times 4$ matrices of rank 4 and
$\Fdot$ be a flag in $\C^8$.
The Schubert variety $X_{3478\,1256}\Fdot$ consists of those $4$-planes $H$ that meet the fixed $4$-plane 
$F_4$ in a subspace of dimension 2.

If $F_4$ is represented as the column space of a $8\times 4$ matrix, then a $4$-plane $H$ from $\calX$ lies
in $X_{3478\,1256}\Fdot$ if and only if
\[
   \mbox{rank}\, \bigl( H\ F_4 \bigr)\ \leq\ 6\,.
\]
{\it A priori}, each of the $7\times 7$ minors of this matrix must vanish for a total of 64 quartic
equations in the entries of $H\in\calX$.
In~\cite[\S~1.3]{HaHS} the Pl\"ucker embedding of the Grassmannian is used to give a smaller set of
equations which are linear combinations of the maximal $4\times 4$ minors of $\calX$.
The dimension of the linear span of such equations is the cardinality of the set 
$\{p\in\binom{[8]}{4}\mid p\not\leq 3478\}$ of increasing sequences $p$ of
length $4$ from $[8]=\{1,\dotsc,8\}$ where one of the inequalities $p_1\leq 3, p_2\leq 4, p_3\leq 7$, or 
$p_4\leq 8$ does not hold.
There are seventeen such sequences
\begin{center}
  5678\,,\ 
  4678\,,\ 
  3678\,,\ 
  4578\,,\ 
  2678\,,\ 
  3578\,,\ 
  4568\,,\ 
  1678\,,\ 
  2578\,,\\
  3568\,,\ 
  4567\,,\ 
  1578\,,\ 
  2568\,,\ 
  3567\,,\ 
  1568\,,\ 
  2567\,,\ 
  1567\,,
\end{center}
so that in Steifel coordinates,  $X_{3478\,1256}\Fdot$ is defined by 17 equations.

The primal-dual formulation uses a variant of the classical reduction to the diagonal.
Consider the map $\perp$ on $G(4,8)$ which sends a linear subspace $H$ to its annihilator, $H^\perp$.
This is an isomorphism in which $\perp(X_w\Fdot)=X_{w^\perp}\Fpdot$, where $\Fpdot$ is the flag of linear
forms annihilating the linear subspaces in $\Fdot$ and $w^\perp=w_0ww_0$, where $w_0(i)=n{+}1{-}i$.

To understand this in Steifel coordinates, pick a basis corresponding to the rows of a matrix whose dual
basis corresponds to the columns.
The dual Schubert variety $X_{w^\perp}\Fpdot$ has Steifel coordinates $\calX_{w^\perp}$, where we send
$K\in \calX_{w^\perp}$ to the row span of $K^T\Phi^{-1}$, where $\Phi$ is the matrix whose first $i$
columns span $F_i$.
In this formulation, the intersection of $X_{3478\,1256}\Fdot$ with the set parametrized by $\calX$ 
is the intersection of the graph of $\perp$ with the product
$X_{3478\,1256^\perp}\Fpdot\times\calX$.
Since $3478\,1256^\perp=3478\,1256$, the primal-dual formulation uses the coordinates
$\calX_{3478\,1256}\times\calX$ with the equations
\[
    K^T \Phi^{-1} H\ =\ 0_{4\times 4}\,,
\]
which state that the four-plane $K^T \Phi^{-1}$ annihilates $H$.
This involves $12=\dim \calX_{3478\,1256}$ new coordinates and 16 bilinear equations, which are the
entries of the matrix $K^T \Phi^{-1} H$.

Finally, the lifted formulation uses the local coordinates
\[
   \calY\ =\ \left(\begin{array}{cccc}1&0&y_{1,1}&y_{1,2}\\
                                      0&1&y_{2,1}&y_{2,2}\end{array}\right)^T
\]
from $\Gr(2,4)$:
For $Y\in \calY$ and $H\in\calX$, the $8\times 2$ matrix $HY$ is a two-plane in $H$.

If $\phi_1,\dotsc,\phi_4$ are the equations the define $F_4$, then the
lifted formulation for  the intersection of $X_{3478\,1256}\Fdot$ with the set parametrized by $\calX$
uses the coordinates $\calX\times\calY$ and the equations
\[
  \phi_i(HY)\ =\ 0\ \qquad\mbox{for } i=1,\dotsc,4\,.
\]
This involves $4=\dim\calY$ new coordinates and 8 bilinear equations (linear in the the entries of $YH$)
as $\phi_i(HY)= 0$ gives two equations, one for each column in $HY$.

\subsection{Added variables for Schubert varieties on flag manifolds in $\C^9$.}

The square primal-dual formulation of a Schubert variety on the flag
manifold~\cite{HaHS} uses that every flag $\Edot$ in $\C^n$ has an annihilating dual flag \defcolor{$\Edot^\perp$}
in the dual space to $\C^n$.
If $\Edot$ has type $\adot$, then $\Edot^\perp$ has type $\defcolor{\apdot}:=\{n{-}a_j\mid a_j\in\adot\}$.
This duality gives an isomorphism $\perp\colon \Fl(\adot;n)\to\Fl(\apdot;n)$ with
${\perp}(X_w\Fdot)=X_{w^\perp}\Fdot^\perp$ (we refer to Section 4 of~\cite{HaHS} where $w^\perp$ is defined).
A variant of the classical reduction to the diagonal allows us to formulate membership of a flag $\Edot$ in
$X_w\Fdot$ by parametrizing $X_{w^\perp}\Fdot^\perp$, using $\ell(w)$ new variables.

As explained in~\cite[Rem.~4.10]{HaHS}, sometimes membership of a flag $\Edot$ of type $\adot$ in a Schubert
variety $X_w\Fdot$ is equivalent to the membership of a projection $\pi(\Edot)$ in the projected Schubert
variety $\pi(X_w\Fdot)=X_v\Fdot$, where $\pi\colon \Fl(\adot;n)\to\Fl(\bdot;n)$ is the natural projection and
$\bdot\subset\adot$. 
When this occurs, the primal-dual formulation uses fewer, $\ell(v)$, new variables.
This is the \demph{reduced primal-dual formulation}.

For every Schubert variety $X_w\Fdot$ on a flag manifold $\Fl(\adot;9)$ with 
$1<|w|<\frac{1}{2}\dim(\adot)$ we compared the numbers of new variables needed in the two formulations.
The restriction $1<|w|$ is because the determinantal formulation when $|w|=1$ is already a
complete intersection. 
The restriction $|w|<\frac{1}{2}\dim(\adot)$ is because, as in Remark~\ref{R:Square_Formulation}, if
$|w|\geq\frac{1}{2}\dim(\adot)$, then we would work in local Steifel coordinates $\calX_w$ for the Schubert
variety $X_w\Fdot$ in any Schubert problem involving $w$ (and any Schubert problem has at most one
permutation satisfying this inequality).

There are $3,395,742$ such Schubert varieties in the 256 flag manifolds $\Fl(\adot;9)$.
We compared the reduced lifted formulation of Subsection~\ref{SS:improved} with the reduced primal-dual
formulation for all these Schubert varieties.
In $141,256$ ($4.160\%$) the primal-dual formulation used fewer new variables, 
in $3,161,233$ ($93.094\%$) the lifted formulation used fewer new variables, and in 
$93,253$ ($2.746\%$) the two were tied.

This overstates the efficiency of the primal-dual formulation.
For example, in only 7 of 1725 relevant Schubert varieties in $\Fl(2,3,5;9)$ did the reduced primal-dual
formulation involve fewer additional variables.
In contrast, on the isomorphic dual flag variety $\Fl(4,6,7;9)$ in 124 out of 1725 relevant Schubert
varieties the reduced primal-dual formulation involved fewer variables.

To gain an idea of how this might be exploited, we determined which of each pair
of dual flag manifolds $\Fl(\adot;9)$ and $\Fl(\apdot;9)$ was more favorable for 
the reduced lifted formulation of its Schubert varieties.
We redid our computation comparing the two formulations, but restricted it to those flag manifolds
$\Fl(\adot;9)$ where $\Fl(\adot;9)$ was more favorable for the reduced lifted formulation
than $\Fl(\apdot;9)$.
This is a fair restriction, for the number of additional variables in the reduced primal-dual
formulation is the same for a Schubert variety and for its dual, but may be different for the 
reduced lifted formulations.

Redoing the computation, there were $1,877,752$ Schubert varieties, as we only considered one
of each dual pair of flag manifolds.
In $53,698$ ($2.860\%$) the primal-dual formulation used fewer new variables, 
in $1,784,646$ ($95.04\%$) the lifted formulation used fewer new variables, and in 
$39,408$ ($2.099\%$) the two were tied.

The reduced lifted formulation is always better for the Grassmannian $\Gr(k,n)$ than for its dual
$\Gr(n{-}k,n)$ when $2k\leq n$.

\begin{lemma}
 If $2k\leq n$, then the reduced lifted fromulation always uses fewer variables than the primal-dual
 formulation for Schubert varieties $X_w\Fdot$ in the Grassmannian $\Gr(k,n)$ with
 $|w|<\frac{1}{2}(k(n{-}k)$. 
\end{lemma}

\begin{proof}
 The original lifted formulation for Schubert varieties in the Grassmannian $\Gr(k,n)$ used
 $\binom{k}{2}$ additional variables, while the primal dual formulation for $X_w\Fdot$ uses 
 $\ell(w)=k(n{-}k)-|w|$ variables.
 The lemma follows as $\ell(w)\geq \frac{1}{2}k(n{-}k)>\frac{1}{2}k(k{-}1)$.
\end{proof}

\begin{remark}
The Grassmannian $\Gr(k,n)$ has a more efficient primal-dual formulation that uses the 
Steifel coordinates of Remark~\ref{R:ReductionXw1nXw2} for the intersection of two Schubert varieties.
This involves $k(n{-}k)-|w_1|-|w_2|$ new variables, while the lifted formulation uses 
$k(k{-}1)$ new variables to formulate membership in two Schubert varieties.
The lifted formulation is more efficient when 
\[
   k(n{-}k)-|w_1|-|w_2| \ >\ k(k-1)\,.
\]
Since we may assume that $|w_1|+|w_2|<\frac{1}{2}k(n{-}k)$, the lifted formulation 
is always more erficient when $k<(n{+}2)/3$ for then
\[
  k(n{-}k)-|w_1|-|w_2| \ \geq\ \frac{1}{2}k(n{-}k)\ >\ \frac{1}{2}k(2k-2)\ =\ k(k{-}1)\,.
  \eqno{\diamond}
\]
\end{remark}

\subsection{Computational time and resources}

We computed instances of three Schubert problems using the (reduced) lifted formulation.
Two were computed using a primal-dual formulation in~\cite{HaHS}, and the third is a
problem with many more solutions. 
In all, the lifted formulation used fewer variables and less computational resources.

\begin{example}\label{Ex:Grassmannian437}
 Consider the Schubert problem in $\Gr(3;9)$ given by the permutations
 \[
    w_1,\dotsc,w_4\ =\ 489\,123567 \qquad \mbox{and} \qquad w_5,\dotsc,w_{10}\ =\ 689\,123457\,.
  \]
 This has $437$ solutions and asks for the $3$-planes in $\C^9$ which nontrivially meet four
 given $4$-planes and six given $6$-planes. 
 The classical formulation of the intersection~\eqref{Eq:SchubertProblem} in Stiefel coordinates for
 $X_{w_1}\Fdot^1\cap X_{w_2}\Fdot^2$ is a system of $12$ variables, $20$ independent linear
 combinations of cubic minors and six cubic determinants. 

 The square primal-dual formulation with similar coordinates involves $24$ variables, $18$ bilinear equations,
 and six cubic determinants. 
 The determinants correspond to the conditions $w_5,\dotsc,w_{10}$ as $|689\,123457|=1$.
 In~\cite{HaHS} we used Bertini~\cite{bertini} to solve an instance of this Schubert problem given by random
 real flags. 
 This computation consumed $20.37$ gigaHertz-hours to calculate $437$ approximate solutions.
 We then used rational arithmetic in alphaCertified~\cite{alphaC} to certify the solutions, which used
 $2.00$ gigaHz-hours. 

 We formulate this Schubert problem using the lifted formulation.
 We use Stiefel coordinates for $X_{w_1}\Fdot^1\cap X_{w_2}\Fdot^2$ which use
 $\dim(\Gr(3;9))-|w_1|-|w_2|=18-3-3=12$ variables. 
 The reduced lifted formulations of $X_{w_i}\Fdot^i$ for $i=3,4$ require a total of
 $|\beta(w_3)|+|\beta(w_4)|=2 + 2=4$ new variables and $2\cdot 5=10$ bilinear equations.
 As in the primal-dual formulation, we formulate membership in the six remaining hypersurface Schubert
 varieties using six cubic determinants. 
 The result is a system of $12+4=16$ variables and $10+6=16$ equations.
 To compare with the primal-dual formulation, we solved a random instance using regeneration with
 the same variables, hardware, software, and software version. 
 The lifted formulation of $16$ variables and equations was a significant improvement, using only $4.75$
 gigaHertz-hours to calculate $437$ approximate solutions.
 The output suggests $107$ of the solutions are real, while the rest are non-real.
 We used seven processors in parallel, but many more could be efficiently used as the regeneration tracked up to
 $2,265$ paths in one step.
 
 Certification time was also significantly improved by using this formulation.
 A $33.54$ gigaHertz-minute computation in alphaCertified~\cite{alphaC} using rational arithmetic
 verified that the $437$ points in the output are indeed approximate solutions and that the corresponding
 exact solutions are distinct. 
 This computation also proved the reality for $107$ of the exact solutions. 
 \hfill$\diamond$
\end{example}

We compare the primal-dual and lifted formulations in a more general flag manifold.

\begin{example}\label{Ex:flag128}
 Consider the Schubert problem with $128$ solutions in $\Fl(2,4,5;8)$ given by
 \begin{eqnarray*}
  w_1         &=& 48\,57\,3\,126\,,\\
  w_2,w_3     &=& 78\,45\,3\,126\,,\\
  w_4,w_5     &=& 68\,57\,4\,123\,,\\
  w_6,w_7,w_8 &=& 78\,46\,5\,123\,,\ \mbox{ and}\\
  w_9         &=& 47\,38\,5\,126\,.
 \end{eqnarray*}
 Applying all improvements given in~\cite{HaHS} produced a primal-dual formulation with $41$ variables, $36$
 bilinear equations, two quadratic determinantal equations corresponding to the hypersurface conditions
 $w_4,w_5$, and three quartic determinantal equations from $w_6,w_7,w_8$. 
 This square system corresponding to a random choice of nine real flags took $2.95$
 gigaHertz-days of processing power to solve and $1.78$ gigaHertz-hours to certify.
 
 We analyze this Schubert problem using a reduced lifted formulation in the Stiefel coordinates
 $\calX_{w_9}$ consisting of $\ell(w_9)=16$ variables. 
 The reduced lifted formulations of $X_{w_i}\Fdot^i$ for $i=1,2,3$ add
 $|\beta(w_1)|+|\beta(w_2)|+|\beta(w_3)|=5+6+6=17$ new variables and $10+9+9=28$ bilinear equations. 
 As in the primal-dual formulation, we formulate $X_{w_i}\Fdot^i$ for $i=4,\dotsc,8$ using two quadrtic and
 three quartic determinants.
 The reduced lifting uses $16+17=33$ variables and $28+2+3=33$ equations.
 As in Example~\ref{Ex:Grassmannian437}, we compare this with the primal-dual formulation using the
 tools which were utilized in~\cite{HaHS}. 
 To facilitate certification, we computed approximate solutions with two extra digits of precision
 compared to our computation in~\cite{HaHS}. 

 With these tighter parameters, we still observed an improvement in efficiency when solving a system with
 the new formulation of $33$ variables and equations. 
 This used $1.13$ gigaHertz-days of computing; less than half the power consumed by the similar instance
 using the primal-dual formulation. 
 The output was $128$ approximate solutions, of which $42$ appeared to be real.
 Again, we used alphaCertified with rational arithmetic to certify the approximate solutions, verify they
 correspond to distinct solutions, and prove that $42$ exact solutions are real. 
 Certification required $1.68$ gigaHertz-hours of processor power, marginally less than certification for
 the similar instance we solved via a primal-dual formulation. 
 
 The initial computation used six processors in parallel, but many more could be efficiently used as the
 regeneration tracked up to 
 $708$ paths in one step.
 Certification could have efficiently used $128$ processors.
 \hfill$\diamond$
\end{example}

We formulated and solved a higher-degree problem in a Grassmannian.

\begin{example}\label{Ex:Grassmannian28490}
\newcommand{\ts}{\hspace{0.5pt}}
\newcommand{\tns}{\hspace{-0.5pt}}
 Consider the Schubert problem with $28,490$ solutions in $\Gr(3;10)$ given by
\[
  w_1,w_2,w_3 = 5\ts9\ts1\tns0\ts\,1\ts2\ts3\ts4\ts6\ts7\ts8\qquad\mbox{and}\qquad
  w_4,\dotsc,w_{15} = 7\ts9\ts1\tns0\ts\,1\ts2\ts3\ts4\ts5\ts6\ts8\,.
\]
 This asks for the $3$-planes in $\C^{10}$ that nontrivially meet three given $5$-planes and twelve given
 $7$-planes. 
 In the determinantal formulation, we parametrize $X_{w_1}\Fdot^1\cap X_{w_2}\Fdot^2$ using
 $\dim(\Gr(3;10))-|w_1|-|w_2|=21-3-3=15$ variables, and membership in $X_{w_3}\Fdot^3$ is given by the
 vanishing of ten independent linear combinations of cubic minors.
 Including the cubic determinants for $X_{w_i}\Fdot^i$ for $i=4,\dotsc,15$ uses $15$ variables and $22$ cubic
 equations.

 The primal-dual formulation uses $33$ variables, $21$ bilinear equations, and twelve cubic determinants.
 The lifted formulation begins with Steifel coordinates involving $15$ variables that parametrize
 $X_{w_1}\Fdot^1\cap X_{w_2}\Fdot^2$. 
 The reduced lifted formulation for $X_{w_3}\Fdot^3$ uses five bilinear equations and adds $|\beta(w_3)|=2$
 variables for a total of $17$ variables. 
 The twelve hypersurface conditions $w_4,\dotsc,w_{15}$ are each given by a single cubic determinant for a
 total of $5+12=17$ equations. 
 The only difference is for $X_{w_3}\Fdot^3$ which use $18$ variables and $21$ bilinear equations with the
 primal-dual formulation but only two variables and five bilinear equations for the lifted formulation.

 We chose $15$ random real flags and solved the corresponding instance of the Schubert problem using
 $1.71$ gigaHertz-months of processing power to apply regeneration in Bertini v.\ 1.4 and $4.00$
 gigaHertz-hours of power to apply four Newton iterations to the output using alphaCertified. 
 This produced $28,490$ approximate solutions, and $1,436$ appeared to be real.
 The main calculation in Bertini used $8$ processors in parallel, but many more could be used efficiently
 as the regeneration tracked up 
 to $148,161$ paths in one step.

 Due to the size of the output, we soft certified our results using floating-point arithmetic in alphaCertified with $192$-bit precision.
 This heuristically verified that the $28,490$ points are approximate solutions, that they correspond to distinct solutions, and that $1,436$ of them correspond to real solutions.
 This computation consumed $47.15$ gigaHertz-minutes of processing power.
 A rigorous computation using rational arithmetic, but only for the $1,436$ apparently real solutions,
 used $1.80$ gigaHertz-days and proved that $1,436$ points in the output are approximate solutions
 corresponding to distinct real solutions. 
 \hfill$\diamond$
\end{example}

We give additional details for the computations and comparisons in
Examples~\ref{Ex:Grassmannian437},~\ref{Ex:flag128}, and~\ref{Ex:Grassmannian28490} at the following site. 

\url{http://www.unk.edu/academics/math/_files/square.html}

\providecommand{\bysame}{\leavevmode\hbox to3em{\hrulefill}\thinspace}
\providecommand{\MR}{\relax\ifhmode\unskip\space\fi MR }
\providecommand{\MRhref}[2]{%
  \href{http://www.ams.org/mathscinet-getitem?mr=#1}{#2}
}
\providecommand{\href}[2]{#2}

\end{document}